\documentclass[12pt]{amsart}
\usepackage[active]{srcltx}
\usepackage{calc,amssymb,amsthm,amsmath,amscd, eucal,ulem}
\usepackage{alltt}
\usepackage[left=1.35in,top=1.25in,right=1.35in,bottom=1.25in]{geometry}
\RequirePackage[dvipsnames,usenames]{color}

\input{kmacros3.sty}
\input{xy}
\xyoption{all}
\usepackage{tikz}

\numberwithin{equation}{theorem}

\renewcommand{\m}{\mathfrak{m}}
\renewcommand{\n}{\mathfrak{n}}

\DeclareMathOperator{\depth}{depth}

\usepackage{fullpage}

\usepackage{setspace}
\usepackage{hyperref}

\usepackage{enumerate}

\usepackage{graphicx}

\usepackage[all,cmtip]{xy}
\usepackage{verbatim}

\theoremstyle{theorem}

\begin{document}
\title{A Sufficient condition for $F$-purity}
\author{Linquan Ma}
\address{Department of Mathematics\\ University of Michigan\\ Ann Arbor\\ Michigan 48109}
\email{lquanma@umich.edu}
\maketitle
\begin{abstract}
It is well known that nice conditions on the canonical module of a local ring have a strong impact in the study of strong $F$-regularity and $F$-purity. In this note, we prove that if $(R,\m)$ is an equidimensional and $S_2$ local ring that admits a canonical ideal $I\cong\omega_R$ such that $R/I$ is $F$-pure, then $R$ is $F$-pure. This greatly generalizes one of the main theorems in \cite{Enescupseudocanonicalcover}. We also provide examples to show that not all Cohen-Macaulay $F$-pure local rings satisfy the above property.
\end{abstract}

\section{introduction}

The purpose of this note is to investigate the condition that $R$ admits a
canonical ideal $I\cong\omega_R$ such that $R/I$ is $F$-pure. This
condition was first studied in \cite{Enescupseudocanonicalcover} and also in \cite{EnescuAfinitenessconditononlocalcohomology} using pseudocanonical covers. And in \cite{Enescupseudocanonicalcover} it was shown that this implies $R$ is $F$-pure under the additional hypothesis that $R$ is Cohen-Macaulay and $F$-injective. Applying some theory of canonical modules for non Cohen-Macaulay rings as well as some recent results in \cite{SharpFpurehasbigtestelement} and \cite{MalFinitenesspropertyoflocalcohomologyforFpurerings}, we are able to drop both the Cohen-Macaulay and $F$-injective condition: we only need to assume $R$ is equidimensional and $S_2$. We also provide examples to show that not all complete $F$-pure Cohen-Macaulay rings satisfy this condition. In fact, if $R$ is Cohen-Macaulay and $F$-injective, we show that this property is closely related to whether the natural injective Frobenius action on $H_\m^d(R)$ can be ``lifted" to an injective Frobenius action on $E_R$, the injective hull of the residue field of $R$. And instead of using pseudocanonical covers, our treatment uses the anti-nilpotent condition for modules with Frobenius action introduced in \cite{EnescuHochsterTheFrobeniusStructureOfLocalCohomology} and  \cite{SharpGradedAnnihilatorsofModulesovertheFrobeniusSkewpolynomialRingandTC}.

In Section 2 we summarize some results on canonical modules of non Cohen-Macaulay rings. These results are well known to experts. In Section 3 we do a brief review of the notions of $F$-pure and $F$-injective rings as well as some of the theory of modules with Frobenius action, and we prove our main result.

\section{canonical modules of non Cohen-Macaulay rings}

In this section we summarize some basic properties of canonical modules of non-Cohen-Macaulay local rings (for those we cannot find references, we give proofs). All these properties are characteristic free. Recall that the canonical module $\omega_R$ is defined to be a finitely generated $R$-module satisfying $\omega_R^{\vee}\cong H_\m^d(R)$ where ${}^{\vee}$ denotes the Matlis dual. Throughout this section we only require $(R,\m)$ is a Noetherian local ring. We do {\it not} need the Cohen-Macaulay or even excellent condition.

\begin{proposition}[{\it cf.} \cite{AoyamaSomeBasicResultsOnCanonicalModules} or Remark 2.2 (c) in \cite{HochsterHunekeIndecomposable}]
\label{existence of canonical modules}
Let $(R,\m)$ be a homomorphic image of a Gorenstein local ring $(S,\n)$. Then $\Ext_S^{\dim S-\dim R}(R, S)\cong\omega_R$.
\end{proposition}

\begin{lemma}
\label{nonzerodivisor on canonical modules}
Let $(R, \m)$ be a local ring that admits a canonical module $\omega_R$. Then every nonzerodivisor in $R$ is a nonzerodivisor on $\omega_R$.
\end{lemma}
\begin{proof}
First assume $R$ is a homomorphic image of a Gorenstein local ring $S$ and write $R=S/J$. Let $x$ be a nonzerodivisor on $R$. The long exact sequence for $\Ext$ yields: \[\Ext_S^{\height J}(R/xR, S)\rightarrow \Ext_S^{\height J}(R, S)\xrightarrow{x} \Ext_S^{\height J}(R, S). \]
But $\Ext_S^{\height J}(R/xR, S)=0$ because the first non-vanishing $\Ext$ occurs at $\depth_{(J+x)}S=\height (J+x)=\height J+1$. So $x$ is a nonzerodivisor on $\Ext_S^{\height J}(R, S)\cong \Ext_S^{\dim S-\dim R}(R, S)\cong\omega_R$.

In the general case, for a nonzerodivisor $x$ in $R$, if we have $0\rightarrow N\rightarrow \omega_R\xrightarrow{x}\omega_R$, we may complete to get $0\rightarrow \widehat{N}\rightarrow \widehat{\omega_R}\xrightarrow{x}\widehat{\omega_R}$. It is easy to see that $\widehat{\omega_R}$ is a canonical module for $\widehat{R}$ and $x$ is a nonzerodivisor on $\widehat{R}$. Now since $\widehat{R}$ is a homomorphic image of a Gorenstein local ring, we have $\widehat{N}=0$ and hence $N=0$.
\end{proof}

\begin{proposition}[{\it cf.} Corollary 4.3 in \cite{AoyamaSomeBasicResultsOnCanonicalModules} or Remark 2.2 (i) in \cite{HochsterHunekeIndecomposable}]
\label{formation of canonical module commutes with localization}
Let $(R,\m)$ be a local ring with canonical module $\omega_R$. If $R$ is equidimensional, then for every $P\in \Spec{R}$, $(\omega_R)_P$ is a canonical module for $R_P$.
\end{proposition}

\begin{proposition}
\label{canonical ideal has a nonzero divisor}
Let $(R,\m)$ be a local ring with canonical module $\omega_R$. If $R$ is equidimensional and unmixed, then the following are equivalent
\begin{enumerate}
\item There exists an ideal $I\cong\omega_R$.
\item $R$ is generically Gorenstein (i.e., $R_P$ is Gorenstein for every minimal prime of $R$).
\end{enumerate}
Moreover, when the equivalent conditions above hold, $I$ contains a nonzerodivisor of $R$.
\end{proposition}
\begin{proof}
Since $R$ is equidimensional, we know that $\omega_{R_P}\cong (\omega_R)_P$ for every prime ideal $P$ of $R$ by Proposition \ref{formation of canonical module commutes with localization}. Let $W$ be the multiplicative system of $R$ consisting of all nonzerodivisors and let $\Lambda$ be the set of minimal primes of $R$. Since $R$ is equidimensional and unmixed, $W$ is simply the complement of the union of the minimal primes of $R$.

$(1)\Rightarrow (2)$: If we have $I\cong \omega_R$, then for every $P\in \Lambda$, $\omega_{R_P}\cong (\omega_R)_P\cong IR_P\subseteq R_P$. But $R_P$ is an Artinian local ring, $l(\omega_{R_P})=l(R_P)$, so we must have $\omega_{R_P}\cong R_P$. Hence $R_P$ is Gorenstein for every $P\in \Lambda$, that is, $R$ is generically Gorenstein.

$(2)\Rightarrow (1)$: Since $R$ is generically Gorenstein, we know that for $P\in \Lambda$, $\omega_{R_P}\cong R_P$. Now we have: \[ W^{-1}\omega_R\cong \prod_{P\in \Lambda}(\omega_R)_P\cong \prod_{P \in \Lambda}\omega_{R_P}\cong \prod_{P\in \Lambda}R_P\cong W^{-1}R.\]
Therefore we have an isomorphism $W^{-1}\omega_R\cong W^{-1}R$. The restriction of the isomorphism to $\omega_R$ then yields an injection $j$: $\omega_R\hookrightarrow W^{-1}R$ because elements in $W$ are nonzero divisors on $\omega_R$ by Lemma \ref{nonzerodivisor on canonical modules}. The images of a finite set of generators of $\omega_R$ can be written as $r_i/w_i$. Let $w=\prod w_i$, we have $wj$: $\omega_R\hookrightarrow R$ is an injection. So $\omega_R$ is isomorphic to an ideal $I\subseteq R$.

Finally, when these equivalent conditions hold, we know that $W^{-1}I\cong \prod_{P\in\Lambda}R_P$ is free. So $W^{-1}I$ contains a nonzerodivisor. But whether $I$ contains a nonzerodivisor is unaffected by localization at $W$. So $I$ contains a nonzerodivisor.
\end{proof}

\begin{lemma}[{\it cf.} Proposition 4.4 in \cite{AoyamaSomeBasicResultsOnCanonicalModules} or Remark 2.2 (f) in \cite{HochsterHunekeIndecomposable}]
\label{R=Hom}
Let $(R, \m)$ be a local ring with canonical module $\omega_R$. Then $\omega_R$ is always $S_2$, and $R$ is equidimensional and $S_2$ if and only if $R\rightarrow \Hom_R(\omega_R, \omega_R)$ is an isomorphism.
\end{lemma}

\begin{proposition}
\label{canonical ideal has height one}
Let $(R,\m)$ be an equidimensional and unmixed local ring that admits a canonical ideal $I\cong\omega_R$. Then $I$ is a height one ideal and $R/I$ is equidimensional and unmixed.
\end{proposition}
\begin{proof}
By Proposition \ref{canonical ideal has a nonzero divisor}, $I$ contains a nonzerodivisor, so its height is at least one. Now we choose a height $h$ associated prime $P$ of $I$ with $h\geq 2$. We localize at $P$, $PR_P$ becomes an associated prime of $IR_P$. In particular, $R_P/IR_P$ has depth $0$ so $H^0_{PR_P}(R_P/IR_P)\neq 0$.

However, by Proposition \ref{formation of canonical module commutes with localization}, $IR_P$ is a canonical ideal of $R_P$, which has dimension $h\geq 2$. Now the long exact sequence of local cohomology gives \[\to H^0_{PR_P}(R_P)\to H^0_{PR_P}(R_P/IR_P)\to H^1_{PR_P}(IR_P)\to.\]
We have $\depth R_P\geq 1$ ($I$ contains a nonzerodivisor) and $\depth IR_P\geq 2$ (the canonical module is always $S_2$ by Lemma \ref{R=Hom}). Hence $H^0_{PR_P}(R_P)=H^1_{PR_P}(IR_P)=0$. The above sequence thus implies $H^0_{PR_P}(R_P/IR_P)=0$ which is a contradiction.

Hence we have shown that every associated prime of $I$ has height one. Since $R$ is equidimensional, this proves $I$ has height one and $R/I$ is equidimensional and unmixed.
\end{proof}

\begin{proposition}[{\it cf.} Page 531 in \cite{HochsterCanonicalelementsinlocalcohomologymodules}]
\label{top local cohomology of canonical module is iso to E}
Let $(R,\m)$ be a local ring of dimension $d$ which admits a canonical module $\omega_R$. Then for every finitely generated $R$-module $M$, $H_\m^d(M)\cong \Hom_R(M, \omega_R)^{\vee}$.
\end{proposition}


\begin{remark}
\label{S_2 implies equidimensional}
\begin{enumerate}
\item When $(R,\m)$ is catenary, $R$ is $S_2$ implies $R$ is equidimensional. Hence, if we assume $R$ is excellent, then in the statement of Lemma \ref{R=Hom} and Proposition \ref{top local cohomology of canonical module is iso to E}, we don't need to assume $R$ is equidimensional.
\item For example, when $(R,\m)$ is a complete local domain, then both canonical modules and canonical ideals exist. And the canonical ideal must have height one and contains a nonzerodivisor.
\end{enumerate}
\end{remark}

\section{main result}

In this section we will generalize greatly one of the main results in \cite{Enescupseudocanonicalcover}. Throughout this section, we always assume $(R,\m)$ is a Noetherian local ring of equal characteristic $p>0$.

We first recall that a map of $R$-modules $N\rightarrow N'$ is {\it pure} if for every $R$-module $M$ the map $N\otimes_RM\rightarrow N'\otimes_RM$ is injective. A local ring $(R,\m)$ is called {\it $F$-pure} if the Frobenius endomorphism $F$: $R\rightarrow R$ is pure. The Frobenius endomorphism on $R$ induces a natural Frobenius action on each local cohomology module $H_\m^i(R)$ (see Discussion 2.2 and 2.4 in \cite{EnescuHochsterTheFrobeniusStructureOfLocalCohomology} for a detailed explanation of this). We say a local ring is {\it $F$-injective} if $F$ acts injectively on all of the local cohomology modules of $R$ with support in $\m$. We note that $F$-pure implies $F$-injective \cite{HochsterRobertsFrobeniusLocalCohomology}.

We will also use some notations introduced in \cite{EnescuHochsterTheFrobeniusStructureOfLocalCohomology} (see also \cite{MalFinitenesspropertyoflocalcohomologyforFpurerings}). We say an $R$-module $M$ is an
{\it $R\{F\}$-module} if there is a Frobenius action $F$: $M\rightarrow M$ such that for all $u\in M$, $F(ru)=r^pu$. We say
$N$ is an {\it $F$-compatible} submodule of $M$ if $F(N)\subseteq N$. We say an $R\{F\}$-module $W$ is {\it anti-nilpotent} if for every $F$-compatible submodule $V\subseteq W$, $F$ acts injectively on $W/V$.

One of the main results in \cite{MalFinitenesspropertyoflocalcohomologyforFpurerings} is the following:
\begin{theorem}[{\it cf.} Theorem 3.8 in \cite{MalFinitenesspropertyoflocalcohomologyforFpurerings}]
\label{F-pure implies anti-nilpotent}
If $(R,\m)$ is $F$-pure, then $H_{\m}^i(R)$ is anti-nilpotent for every $i$.
\end{theorem}
We also recall the following result of Sharp in \cite{SharpFpurehasbigtestelement}:
\begin{theorem}[{\it cf.} Theorem 3.2 in \cite{SharpFpurehasbigtestelement}]
\label{sharp's theorem} A local ring $(R,\m)$ is $F$-pure if and
only if $E_R$ has a Frobenius action compatible with its $R$-module structure that is torsion-free
(injective).
\end{theorem}

We will also need the following lemma:
\begin{lemma}
\label{lemma on non-injectivity of top local cohomology}
Let $(R,\m)$ be an equidimensional local ring of dimension $d$ that admits a canonical module $\omega_R$. Let $I$ be a height one ideal of $R$ that
contains a nonzerodivisor.  Then $H^d_\m(I) \to H^d_\m(R)$ induced by $I\hookrightarrow R$ is not injective.
\end{lemma}
\begin{proof}
By Proposition \ref{top local cohomology of canonical module is iso to E}, to show $H^d_\m(I) \to H^d_\m(R)$ is not injective, it suffices to show
\begin{equation}
\label{dual statement for surjectivity}
\Hom_R(R, \omega_R)\to \Hom_R(I, \omega_R)
\end{equation}
is not surjective.

It suffices to show (\ref{dual statement for surjectivity}) is not surjective after we localize at a height one minimal prime $P$ of $I$. Since $I$ contains a nonzerodivisor and $P$ is a height one minimal prime of $I$, it is straightforward to see that $R_P$ is a one-dimensional Cohen-Macaulay ring with $IR_P$ a $PR_P$-primary ideal. And by Proposition \ref{formation of canonical module commutes with localization}, $(\omega_R)_P$ is a canonical module of $R_P$. Hence to show $\Hom_R(R, \omega_R)_P\to \Hom_R(I, \omega_R)_P$ is not surjective, we can apply Proposition \ref{top local cohomology of canonical module is iso to E} (taking Matlis dual of $E_{R_P}$) and we see it is enough to prove that \[H^1_{PR_P}(IR_P)\to H^1_{PR_P}(R_P)\] is not injective. But this is obvious because we know from the long exact sequence that the kernel is $H^0_{PR_P}(R_P/IR_P)$, which is nonzero because $I$ is $PR_P$-primary.
\end{proof}

The following result was first proved in \cite{Enescupseudocanonicalcover} using pseudocanonical covers
under the hypothesis that $R$ be Cohen-Macaulay and $F$-injective (see Corollary 2.5 in \cite{Enescupseudocanonicalcover}). We want to drop these conditions and only assume $R$ is equidimensional and $S_2$ (as in Remark \ref{S_2 implies equidimensional}, when $R$ is excellent, we only need to assume $R$ is $S_2$). Our argument here is quite different. Here is our main result:

\begin{theorem}
\label{Florian's condition implies F-pure} Let $(R,\m)$ be an equidimensional and $S_2$ local ring of dimension $d$ which admits a canonical ideal $I\cong \omega_R$ such that $R/I$ is $F$-pure. Then $R$ is F-pure.
\end{theorem}

\begin{proof}

First we note that $I$ is a height one ideal by Proposition \ref{canonical ideal has height one}. In particular we know that $\dim R/I< \dim R=d$. We have a short exact sequence:
\[ 0\rightarrow I\rightarrow R\rightarrow R/I\rightarrow 0. \]
Moreover, if we endow $I$ with an $R\{F\}$-module structure induced
from $R$, then the above is also an exact sequence of
$R\{F\}$-modules. Hence the tail of the long exact sequence of local
cohomology gives an exact sequence of $R\{F\}$-modules, that is, a
commutative diagram (we have $0$ on the right because $\dim R/I< d$):
\[  \xymatrix{
   H_\m^{d-1}(R) \ar[d]^{F} \ar[r]^{\varphi_1} & H_\m^{d-1}(R/I)  \ar[r]^-{\varphi_2} \ar[d]^{F}&     H_\m^d(I)  \ar[d]^{F} \ar[r]^{\varphi_3}  & H_\m^d(R) \ar[d]^{F} \ar[r] & 0\\
   H_\m^{d-1}(R) \ar[r]^{\varphi_1} & H_\m^{d-1}(R/I)  \ar[r]^-{\varphi_2} &     H_\m^d(I)  \ar[r]^{\varphi_3}  & H_\m^d(R) \ar[r] & 0
} \]where the vertical maps denote the Frobenius actions on each
module.

Since $R$ is equidimensional and $S_2$, we know that
$H_\m^d(I)\cong H_\m^d(\omega_R)\cong \Hom_R(\omega_R,\omega_R)^\vee\cong R^\vee\cong E_R$ by Proposition \ref{top local cohomology of canonical module is iso to E} and Lemma \ref{R=Hom}. We want to show that, under the hypothesis, the Frobenius action on $H_\m^d(I)\cong E_R$ is injective. Then we will be done by Theorem \ref{sharp's theorem}.

Suppose the Frobenius action on $H_\m^d(I)$ is not injective, then the nonzero socle element $x\in H_\m^d(I)\cong E_R$ is in the kernel, i.e., $F(x)=0$. From Proposition \ref{canonical ideal has height one} and Lemma \ref{lemma on non-injectivity of top local cohomology} we know that $\varphi_3$ is not injective. So we also have $\varphi_3(x)=0$. Hence $x=\varphi_2(y)$ for some $y\in H_\m^{d-1}(R/I)$. Because $0=F(x)=F(\varphi_2(y))=\varphi_2(F(y))$,
we get that $F(y)\in \im \varphi_1$. Using the commutativity of the diagram, it is straightforward to check that $\im\varphi_1$ is an $F$-compatible submodule of $H_\m^d(R/I)$. Since $R/I$ is $F$-pure, $H_\m^{d-1}(R/I)$ is anti-nilpotent by Theorem \ref{F-pure implies anti-nilpotent}. Hence $F$ acts injectively on $H_\m^{d-1}(R/I)/\im\varphi_1$. But clearly $F(\overline{y})=\overline{F(y)}=0$ in $H_\m^{d-1}(R/I)/\im\varphi_1$, so $\overline{y}=0$. Therefore $y\in\im\varphi_1$. Hence $x=\varphi_2(y)=0$ which is a contradiction because we assume $x$ is a nonzero socle element.
\end{proof}


\begin{remark}
\label{Cohen-Macaulay case}
If we assume that $R$ is Cohen-Macaulay and $F$-injective in Theorem \ref{Florian's condition implies
F-pure}, then the diagram used in the proof of Theorem \ref{Florian's condition implies F-pure} reduces to the following:
\[  \xymatrix{
   0 \ar[r] & H_\m^{d-1}(R/I)  \ar[r] \ar[d]^{F}&     H_\m^d(I)  \ar[d]^{F} \ar[r]  & H_\m^d(R) \ar[d]^{F} \ar[r] & 0\\
   0 \ar[r] & H_\m^{d-1}(R/I)  \ar[r] &     H_\m^d(I)  \ar[r]  & H_\m^d(R) \ar[r] & 0
} \] Since $R$ is $F$-injective, the Frobenius action on $H_\m^d(R)$
is injective. So this diagram and the five lemma tell us
immediately that the Frobenius action on $E_R\cong H_\m^d(I)$ is
injective if and only if the Frobenius action on $H_\m^{d-1}(R/I)$
is injective, i.e., if and only if $R/I$ is $F$-injective (or equivalently, $F$-pure
since when $R$ is Cohen-Macaulay, $R/I$ is Gorenstein). This gives a quick proof of Enescu's original result.
\end{remark}

It is quite natural to ask, when $R$ is an $F$-pure Cohen-Macaulay ring and has a canonical
module, can we always find $I\cong\omega_R$ such that $R/I$ is $F$-pure? Note that by Proposition \ref{canonical ideal has a nonzero divisor}, in this situation $R$ has a canonical ideal $I\cong\omega_R$ because $R$ is $F$-pure, hence reduced, in particular generically Gorenstein.

However the following example shows that this is not always true. So in view of Remark \ref{Cohen-Macaulay case}, even when $R$ is Cohen-Macaulay and $F$-pure, the injective Frobenius action on $E_R$ may {\it
not} be compatible with the natural Frobenius action $H_\m^d(R)$
under the surjection $E_R\cong H_\m^d(I)\twoheadrightarrow H_\m^d(R)$, no matter
how one picks $I\cong\omega_R$. I would like to thank Alberto F. Boix for pointing out to me that this example has been studied by Goto in \cite{GotoAproblemonNoetherianlocalringsofcharacteristicp}.

\begin{example}[{\it cf.} Example 2.8 in \cite{GotoAproblemonNoetherianlocalringsofcharacteristicp}]
\label{Florian's condition is stronger than F-pure} Let
$R=K[[x_1,\dots,x_n]]/(x_ix_j, i\neq j)$ where $n\geq 3$. Then $R$
is a $1$-dimensional complete $F$-pure non-Gorenstein Cohen-Macaulay local ring. So $R/I$
will be a $0$-dimensional local ring (non-Gorenstein property ensures that $I$ is not the unit ideal). If it is $F$-pure, it must be a field (since $F$-pure implies reduced). So $R/I$ is $F$-pure if and only if
$\omega_R\cong I \cong \m$. But clearly $\omega_R \neq \m$, because
one can easily compute that the type of $\m$ is $n$: $x_1+\dots+x_n$
is a regular element, and each $x_i$ is in the socle of
$\m/(x_1+\dots+x_n)\m$.
\end{example}

Last we point out a connection between our main theorem and some theory in $F$-adjunction. In fact, results of Schwede in \cite{SchwedeFAdjunction} imply that if $(R,\m)$ is an $F$-finite normal local ring with a canonical ideal $I\cong\omega_R$ which is principal in codimension 2 and $R/I$ is normal and $F$-pure, then $R$ is $F$-pure (take $X=\Spec R$, $\Delta=0$ and $D=-K_R$ in Proposition 7.2 in \cite{SchwedeFAdjunction}). The argument in \cite{SchwedeFAdjunction} is geometrical and is in terms of Frobenius splitting. Our Theorem \ref{Florian's condition implies F-pure} is a natural generalization (we don't require any $F$-finite, normal or principal in codimension 2 conditions) and we use the dualized argument, i.e., studying the Frobenius actions on local cohomology modules. I would like to thank Karl Schwede for pointing out this connection to me.

\section*{Acknowledgement}
I would like to thank Mel Hochster for many helpful and valuable discussions on the problem. I would like to thank Mel Hochster and Craig Huneke for suggesting the proof of Lemma \ref{lemma on non-injectivity of top local cohomology} used here. I am grateful to Alberto F. Boix, Florian Enescu, Rodney Sharp, Karl Schwede and Wenliang Zhang for some valuable comments. And I also thank the referee for his/her comments.

\bibliographystyle{skalpha}
\bibliography{CommonBib}
\end{document}